\documentclass[12pt,reqno]{amsart}
\usepackage{amsmath,amssymb,amsthm,mathtools,wasysym,calc,verbatim,tikz,pgfplots,url,mathrsfs,fullpage,bbm}
\numberwithin{equation}{section}
\usepackage{comment}

\newtheorem{mydef}{Definition}[section]
\newtheorem{myrem}{Remark}[section]
\newtheorem{mytheo}{Theorem}[section]
\newtheorem{mylem}{Lemma}[section]
\newtheorem{mycoro}{Corollary}[section]

\newtheorem{myprop}{Proposition}[section]

\usepackage{enumerate}

\theoremstyle{definition}

\theoremstyle{remark}

\usepackage[colorlinks = true,
            linkcolor  = blue,
            citecolor  = blue,
            urlcolor   = blue,
            bookmarks  = true]{hyperref}

\theoremstyle{remark}

\numberwithin{equation}{section}

\usepackage[backend=biber, style=numeric, citestyle=numeric-comp]{biblatex}
\DeclareFieldFormat{labelnumberwidth}{\mkbibbrackets{#1}}
\DeclareNameAlias{author}{family-given}

\DeclareFieldFormat*{title}{#1}

\AtEveryBibitem{%
  \clearfield{issn}
  \clearfield{url}
  \clearfield{eprint}
}

\renewbibmacro{in:}{}

\DeclareFieldFormat[article]{volume}{\textbf{#1}}
\DeclareFieldFormat[article]{number}{\textbf{#1}}
\renewbibmacro*{volume+number+eid}{%
  \printfield{volume}
  \iffieldundef{number}
    {}%
    {\printtext[parens]{\printfield{number}}}
  \setunit{\addcomma\space}%
  \printfield{eid}}

\DeclareFieldFormat[journal]{journaltitle}{\mkbibemph{#1}}

\DeclareNameAlias{author}{family-given}
\DeclareNameAlias{editor}{family-given}

\DeclareDelimFormat{namedelim}{\addcomma\space}
\DeclareFieldFormat[article,incollection,inproceedings,book]{giveninit}{\mkbibnamegiven{#1}}

\DeclareNameFormat{family-given}{%
  \usebibmacro{name:family-given}
    {\namepartfamily}
    {\namepartgiveni} 
    {\namepartprefix}
    {\namepartsuffix}%
  \usebibmacro{name:andothers}}

\DeclareFieldFormat{pages}{#1}

 \allowdisplaybreaks[4]

\begin{document}

\title{Uniform Concentration for $\alpha$-subexponential Random Operators}

\author{Tiankun Diao}
\address{Shandong University,  Jinan, China.}
\email{tiankundiao@gmail.com}

\author{Xuanang Hu}
\address{Shandong University,  Jinan, China.}
\email{xuananghu7@gmail.com}

\author{Vladimir V. Ulyanov}
\address{Research University Higher School of Economics and Lomonosov Moscow State University, Moscow, Russia}
\email{vulyanov@cs.msu.ru }

\author{Hanchao Wang}
\address{ Shandong University, Jinan,  China.}
\email{wanghanchao@sdu.edu.cn}


\begin{abstract}
Random matrices acting on structured sets play a fundamental role in high-dimensional geometry, compressed sensing, and randomized algorithms. Existing results primarily focus on subgaussian models, when random matrices act as near-isometries on sets with optimal tail behavior. Nevertheless, very often in applications we deal with distributions with heavy tails that are not subgaussian but have at least exponential-type tails.

In this work, we study random matrices $A$ whose rows (or columns) have 
$\alpha$-subexponential tail distributions with $\alpha\in(0,2]$. So subgaussian and sub-exponential models are included in as special cases. We establish concentration type inequality for $Ax$, where x belongs to the bounded subsets of 
$\mathbb{R}^n$
, showing that their geometric distortion is governed by Talagrand's functional of the set and depends on the tail parameter $\alpha$
. Our results extend the known optimal inequalities in the subgaussian regime
(
$\alpha=2$
)
, and provide new guarantees for heavier-tailed, yet exponentially integrable, random matrices.

These findings extend the theory of random matrices beyond the subgaussian framework. Moreover, they yield near-isometric embedding results applicable to dimension reduction and allow us to make robust high-dimensional inference under non-Gaussian measurements.

\end{abstract}

\maketitle

\section{Introduction}

Random matrices play a central role in high-dimensional probability, compressed sensing, and randomized numerical linear algebra due to their remarkable ability to preserve geometric structure. A fundamental question in this area is to understand when a random linear map
\(
A \in \mathbb{R}^{m \times n}
\)
acts as a near-isometry on a given subset \(T \subset \mathbb{R}^n\), meaning that Euclidean norms of vectors in \(T\) are approximately preserved under the mapping \(x \mapsto Ax\). Such property underpins classical results including the Johnson--Lindenstrauss lemma and restricted isometry property.

When the random matrix $A$ has Gaussian entries, the near-isometric behavior of $A$  is by
now well understood and is known to be governed by geometric complexity of measures such as Gaussian
width or Gaussian complexity; see, for instance, \cite{vershynin2018high,schechtman2006two}. Extending these results beyond the Gaussian setting has attracted significant attention, motivated by both theoretical considerations and applications in which Gaussian randomness is either impractical or undesirable. Early progress in this direction focused on subgaussian random matrices, where concentration inequalities and chaining techniques were effectively applied; see, e.g., \cite{mendelson2008uniform,liaw2017simple}. There are also studies focusing on non-subgaussian scenarios. For example,  Adamczak et al \cite{adla2} obtained tail estimates for the norms of projections of sums of independent log-concave random vectors, and uniform versions of these in the form of tail estimates for operator norms of matrices and their sub-matrices in the setting of a log-concave ensemble.

Two complementary non-Gaussian frameworks have since emerged as particularly influential. The first considers matrices with independent, isotropic, subgaussian rows. In this row-wise model, sharp deviation bounds for the action of random matrices on sets were established, with optimal dependence on the subgaussian norm and tight tail behavior; see \cite{jeong2022sub}. The second framework focuses on matrices with independent subgaussian columns of fixed Euclidean norm. This column-wise model, developed in \cite{plan2025random}, captures critical phenomena arising in sparse and structured random matrices and reveals behavior that differs fundamentally from row-independent or entry-independent ensembles, including improved dependence on distributional parameters.

All frameworks rely critically on subgaussian assumptions that impose light tails and strong concentration properties. In many applications, however, random data exhibit heavier-tailed behavior that deviates from the subgaussian regime, yet still exhibit subexponential tails. Examples arise naturally in robust statistics, signal processing under impulsive noise, and randomized algorithms based on non-Gaussian sketches. This raises the following question:
\begin{quote}
To what extent are the nearly isometric properties of random matrices acting on sets preserved
when the sub-Gaussian assumptions are relaxed to distributions with exponential tails?
\end{quote}

We now present our main results.

\subsection{Row-wise \(\alpha\)-subexponential model}
Our first main result generalizes the row-wise framework introduced by Jeong  et al. \cite{jeong2022sub} to \(\alpha\)-subexponential random matrices.

\begin{mytheo} [Row-wise model]\label{thm:row}
Let \( B \in \mathbb{R}^{l \times m} \) be a fixed matrix and \( A \in \mathbb{R}^{m \times n} \) be a random matrix with zero mean. Assume $\alpha\in(0,2]$. Let rows of $A$ be independent isotropic and have \(\psi_\alpha\)-norm  (or quasi-norm when $\alpha<1$) bounded by $K$ uniformly. Let \( T \subset \mathbb{R}^n \) be a bounded set. Then
\[
\textsf{E} \sup_{x \in T} \left| \|BAx\|_2 - \|B\|_{HS} \|x\|_2 \right| \leq C(\alpha)K^{4/\alpha}  \|B\|_{op} \left( \gamma_\alpha(T)+rad(T) \right).
\]
Moreover, for any \( u \geq 0 \), with probability at least \( 1 - C \exp(-u^\alpha) \),

\[
\sup_{x \in T} \left| \|BAx\|_2 - \|B\|_{HS} \|x\|_2 \right| \leq C(\alpha)K^{4/\alpha}  \|B\|_{op} \left( \gamma_\alpha(T) + u \cdot \text{rad}(T) \right),
\]
where \( C(\alpha) \) is an absolute constant depending only on \(\alpha\); $\|B\|_{HS}$ and $\|B\|_{op}$ denote Frobenius and operator norm of $B$ respectively;
\(\gamma_\alpha(T)\) is a Talagrand's functional defined in Section~2, and $rad(T)$ is given by $\sup_{t\in T}\|t\|_2$; $\psi_\alpha$ norm (or quasi-norm when $\alpha<1$) is defined by
\[
\|\xi\|_{\psi_\alpha}
:= \inf \left\{ t>0 : \textsf{E}\exp\!\left( \left|\frac{\xi}{t}\right|^{\alpha} \right) \le 2 \right\}.
\]

\end{mytheo}

This theorem extends existing deviation inequalities for subgaussian matrices on sets to the more general class of \(\alpha\)-subexponential distributions, while preserving the correct dependence on the geometry of \(T\).

When $B$ is the identity matrix, we have the following corollary.

\begin{mycoro}\label{coro:row}
    Let \( A \in \mathbb{R}^{m \times n} \) be a random matrix with zero mean. Assume $\alpha\in(0,2]$. Let rows of $A$ be independent isotropic and have \(\psi_\alpha\)-norm  (or quasi-norm when $\alpha<1$) bounded by $K$ uniformly. Let \( T \subset \mathbb{R}^n \) be a bounded set. Then
\[
\textsf{E} \sup_{x \in T} \left| \|Ax\|_2 - \sqrt{m} \|x\|_2 \right| \leq C(\alpha)K^{4/\alpha}   \left( \gamma_\alpha(T)+rad(T) \right).
\]
Moreover, for any \( u \geq 0 \), with probability at least \( 1 - Ce^{-u^\alpha} \),

\[
\sup_{x \in T} \left| \|Ax\|_2 - \sqrt{m} \|x\|_2 \right| \leq C(\alpha)K^{4/\alpha}   \left( \gamma_\alpha(T) + u \cdot \text{rad}(T) \right).
\]
\end{mycoro}

\subsection{Column-wise \(\alpha\)-subexponential model}

Our second main result generalizes the column-wise framework introduced by Plan and Vershynin \cite{plan2025random} to \(\alpha\)-subexponential random matrices.

\begin{mytheo}
[Column-wise model]\label{thm:column}
Let \(A \in \mathbb{R}^{m \times n}\) be a random matrix whose columns \(A_i\) are independent, mean-zero random vectors in \(\mathbb{R}^m\). Assume that \(\|A_i\|_2 = 1\) almost surely and that the columns have \(\psi_\alpha\)-norm (or quasi-norm when $\alpha<1$) bounded by $K$ for some \(\alpha\in(0,2]\). Then,
\[
\textsf{E}\sup_{x \in T}
\left|
\|Ax\|_2 - \|x\|_2
\right|
\;\le\;
C(\alpha)K \, \left(\gamma_\alpha(T)+rad(T)\right).
\]
Moreover, for any \( u \geq 0 \), with probability at least \( 1 - Ce^{-u^\alpha} \), we have
\[
\sup_{x \in T}
\left|
\|Ax\|_2 - \|x\|_2
\right|
\;\le\;
C(\alpha)K \, \left(\gamma_\alpha(T)+ u\cdot rad(T)\right).
\]
\end{mytheo}

Theorem~\ref{thm:column} extends  Theorem 1.3 proved in \cite{plan2025random}. Moreover, our result is based on a fundamentally different approach.

It follows immediately that

\begin{mycoro}\label{coro:column}
    Let \(A \in \mathbb{R}^{m \times n}\) be a random matrix whose columns \(A_i\) are independent, mean-zero random vectors in \(\mathbb{R}^m\). Assume that \(\|A_i\|_2 = \lambda>0\) almost surely and that the columns have \(\psi_\alpha\)-norm (or quasi-norm when $\alpha<1$) bounded by $K$ for some \(\alpha\in(0,2]\). Then
\[
\textsf{E}\sup_{x \in T}
\left|
\|Ax\|_2 - \lambda\|x\|_2
\right|
\;\le\;
C(\alpha)K \, \left(\gamma_\alpha(T)+rad(T)\right).
\]
Moreover, for any \( u \geq 0 \), with probability at least \( 1 - Ce^{-u^\alpha} \), we have
\[
\sup_{x \in T}
\left|
\|Ax\|_2 - \lambda\|x\|_2
\right|
\;\le\;
C(\alpha)K \, \left(\gamma_\alpha(T)+ u\cdot rad(T)\right).
\]
\end{mycoro}

\begin{myrem}[On the necessity of column normalization]
Unlike Corollary~\ref{coro:row}, the present result requires the columns of $A$ to satisfy the strong normalization condition
\[
\|A_i\|_2 = \lambda \quad \text{a.s. for all } i.
\]
This assumption cannot be dropped, even in the one-dimensional case $n=1$, nor can it be replaced by isotropy alone. 
For example, let $b \sim \mathrm{Bernoulli}(1/2)$ and $X$ be uniformly distributed on $\sqrt{m}\,\mathbb{S}^{m-1}$, independent of $b$. 
Then $b X$ is isotropic and subgaussian. Moreover, if the columns of $A$ are i.i.d.\ copies of $b X$ and $T=\{e_1\}$, we have
\[
\sup_{x\in T}\bigl|\|Ax\|_2-\lambda\|x\|_2\bigr|
= |b\sqrt{m}-\lambda|
\ge \tfrac{\sqrt{m}}{2}
\]
with probability $1/2$, for any choice of $\lambda$.
Therefore, without column normalization, the two inequalities in Theorem~\ref{coro:column} cannot be established, as they demand control over $\left|
\|Ax\|_2 - \lambda\|x\|_2
\right|$ that does not depend on  dimension $m$.
\end{myrem}

\subsection{Methodological remarks}

In our opinion, the main achievement of this paper is the method of proving the theorem, which differs from the approach in \cite{plan2025random}. The original argument of Plan and Vershynin relies heavily on fine properties of subgaussian random variables, including sharp tail bounds and moment growth that do not extend naturally to distributions with heavier-tailed. As a result, their method cannot be directly adapted beyond the subgaussian setting.

In contrast, our approach avoids the use of tools specific to subgaussian distributions and is based on a more straightforward decomposition method combined with elementary concentration arguments. The idea applies uniformly to all \(\alpha >0\). At the same time, leads to a more straightforward and more transparent proof even in the subgaussian case.

The paper is organized as follows. We introduce notation and preliminaries in Section~2. Sections~3 and~4 develop the main deviation inequalities for the row-wise and column-wise \(\alpha\)-subexponential models, respectively. Section 5 is devoted to applications, including Johnson--Lindenstrauss embeddings, 
restricted isometry properties of $\alpha$-subexponential random matrices, 
and column-normalized structured matrices. In conclusion, we will discuss the directions for future work.

\section{Preliminaries}
In the sequel, $c, C$ with and without indices denote constants. We assume that these constants do not necessarily have to be the same every time they appear.  If a constant depends on a parameter $\alpha$, we write $c = c(\alpha)$.

Throughout this paper, we use the following standard notation.
For a vector $x = (x_1,\dots,x_n) \in \mathbb{R}^n$ and $1 \le p < \infty$, we denote by
\[
\|x\|_p := \left( \sum_{i=1}^n |x_i|^p \right)^{1/p}
\]
the usual $\ell_p$-norm.
For a real-valued random variable $X$ and $1 \le p < \infty$, its $L^p$-norm is defined as
\[
\|X\|_{L^p} := \big( \textsf{E}|X|^p \big)^{1/p}.
\]
For a matrix $A = (a_{ij})$, we write
\[
\|A\|_{\mathrm{HS}} := \left( \sum_{i,j} a_{ij}^2 \right)^{1/2}
\quad \text{and} \quad
\|A\|_{\mathrm{op}} := \sup \{ \|Ax\|_2 : \|x\|_2 = 1 \}
\]
for the Hilbert--Schmidt norm and the operator norm, respectively.

\subsection{$\alpha$-subexponential Random Variables and Isotropic Vectors}
\begin{mydef}[\texorpdfstring{$\alpha$}{alpha}-subexponential random variables]
Let $ \alpha>0$. We say that a random variable $\xi$ is 
\emph{$\alpha$-subexponential} if there exists an absolute constant $c>0$ such that
for all $t \ge 0$,
\[
\textsf{P}\big( |\xi - \textsf{E}\xi| \ge t \big)
\le 2 \exp\!\left( - \frac{t^{\alpha}}{c} \right).
\]

The corresponding $\psi_\alpha$ norm (or quasi-norm when $\alpha<1$) is defined by
\[
\|\xi\|_{\psi_\alpha}
:= \inf \left\{ t>0 : \textsf{E}\exp\!\left( \left|\frac{\xi}{t}\right|^{\alpha} \right) \le 2 \right\}.
\]
\end{mydef}

We next collect several fundamental properties of $\alpha$-subexponential random variables.
These results are standard and well known, so we omit the proofs.
For a detailed exposition, we refer the reader to Proposition A.1 in \cite{jeong2022sub}.

\begin{myprop}
    \label{Lem_2.1}
	Let $\eta, \xi$ be random variables and  $\alpha >0$.
\begin{enumerate}[(a)]
    \item If $|\eta|\le |\xi|$, then $\|\eta\|_{\psi_\alpha}\le \|\xi\|_{\psi_\alpha}$ .
    
    \item If $\|\eta\|_{\psi_\alpha} \leq K < \infty$, then $\textsf{P}(|\eta| \geq t) \leq 2\exp(-t^\alpha/K^\alpha)$ for all $t \geq 0$.
    
    \item If $\textsf{P}(|\eta| \geq t) \leq 2\exp(-t^\alpha/K^\alpha)$ for all $t \geq 0$ and some $K > 0$, then $\|\eta\|_{\psi_\alpha} \leq C(\alpha)K$.
    
    \item $\|\eta^p\|_{\psi_\alpha} = \|\eta\|_{\psi_{p\alpha}}^p$ for all $p \geq 1$. In particular, $\|\eta^2\|_{\psi_1} = \|\eta\|_{\psi_2}^2$.
    
    \item $\|\eta\xi\|_{\psi_\alpha} \leq \|\eta\|_{\psi_{p\alpha}} \|\xi\|_{\psi_{q\alpha}}$ for $p,q \in (1,\infty)$ such that $\frac{1}{p} + \frac{1}{q} = 1$. In particular, $\|\eta\xi\|_{\psi_1} \leq \|\eta\|_{\psi_2} \|\xi\|_{\psi_2}$.
    
    \item $\textsf{E}|\eta|^p \leq \left(C p^{\frac{1}{\alpha}} \|\eta\|_{\psi_\alpha}\right)^p$ for all $p \geq 1$ and some absolute constant $C \leq 4$.
    
    \item $\|\eta - \textsf{E}\eta\|_{\psi_\alpha} \leq C(\alpha)\|\eta\|_{\psi_\alpha}$.
    
    \item $\|\eta\|_{\psi_\alpha} \leq C\|\eta\|_{\psi_\beta}$ for all $\beta \geq \alpha$ and some absolute constant $C$.

    \item $\|\eta + \xi\|_{\psi_\alpha}\le C(\alpha)\left(\|\eta\|_{\psi_\alpha}+\|\xi\|_{\psi_\alpha}\right).$
\end{enumerate}

\end{myprop}

Given random variables $\xi_1,\dots,\xi_n$, we write
\[
\xi := (\xi_1,\dots,\xi_n)
\]
for the associated random vector in $\mathbb{R}^n$.
The corresponding $\psi_\alpha$ norm (or quasi-norm when $\alpha<1$) of $\xi$ is defined by
\[
\|\xi\|_{\psi_\alpha}
:= \sup_{x \in S^{n-1}} \|\langle \xi, x \rangle\|_{\psi_\alpha}.
\]

We say that a random vector $X$ in $\mathbb{R}^n$ is \emph{isotropic} if
\[
\textsf{E}\, X X^\top = I_n,
\]
where $I_n$ denotes the identity matrix in $\mathbb{R}^n$.

\subsection{Generic chaining and Talagrand's $\gamma$-functionals}

We begin by recalling the notion of an admissible sequence, which plays a central role
in the definition of Talagrand's $\gamma$-functionals.

\begin{mydef}[Admissible sequence]\label{def2.3}
Let $(T,d)$ be a metric space. An admissible sequence of partitions of $T$ is an increasing sequence $\left(\mathcal{A}_n\right)_{n\geq 0}$ of partitions of $T$ such that  
\[
|\mathcal{A}_0| = 1
\qquad\text{and}\qquad
|\mathcal{A}_k| \le 2^{2^k}, \quad k \ge 1.
\]
\end{mydef}

Using admissible sequences, we define Talagrand's $\gamma_2$-functional.
For a metric space $(T,d)$, set
\[
\gamma_2(T,d)
:= \inf \sup_{t \in T} \sum_{n \ge 0} 2^{n/2} \, \Delta_d(A_n(t)),
\]
where the infimum is taken over all admissible sequences $(\mathcal{A}_k)_{k \ge 0}$. Here $A_n\left(t\right)$ denotes the unique set $A\in \mathcal{A}_n$ such that $t\in A$
and 
\[
\Delta_d(A) := \sup_{s,t \in A} d(t,s).
\]

More generally, for $\alpha >0$, the $\gamma_\alpha$-functional of $(T,d)$
is defined by
\[
\gamma_\alpha(T,d)
:= \inf \sup_{t \in T} \sum_{n \ge 0} 2^{n/\alpha} \, \Delta_d(A_n(t)),
\]
where the infimum is again taken over all admissible sequences.
When $d$ coincides with the Euclidean metric $d_2$, we write $\gamma_\alpha(T)$ for simplicity.

We employ the chaining method to describe the concentration behavior of random processes with $\alpha$-subexponential increments. 
To this end, we first recall a fundamental result due to Talagrand.

\begin{mytheo}[See Talagrand {\cite[Theorem 2.7.14]{Talagrand2014}}]\label{thm:talagrand 2.7.14}
Let $(S_t)_{t \in T}$ be a centered stochastic process, not necessarily symmetric. 
For $n \ge 1$, define the increment metric
\[
\delta_n(s,t) := \|S_s - S_t\|_{2^n}.
\]
For a subset $A \subset T$, let $\Delta_n(A)$ denote its diameter with respect to $\delta_n$.
Let $(\mathcal{A}_n)_{n \ge 0}$ be an admissible sequence of partitions of $T$.
Then there exists a universal constant $C>0$ such that
\begin{equation}\label{eq:thm2.7.14-mean}
\textsf{E} \sup_{s,t \in T} |S_s - S_t|
\le C \sup_{t \in T} \sum_{n \ge 0} \Delta_n(A_n(t)).
\end{equation}

Moreover, for any $u>0$, let $k$ be the largest integer such that $2^k \le u^2$. Then
\begin{equation}\label{eq:thm2.7.14-tail}
\textsf{P}\Bigg(
\sup_{s,t \in T} |S_s - S_t|
\ge c \Delta_k(T) + \sup_{t \in T} \sum_{n \ge 0} \Delta_n(A_n(t))
\Bigg)
\le C \exp(-u^2),
\end{equation}
where $c,C>0$ are universal constants.
\end{mytheo}

As a direct consequence, the above theorem yields a sharp bound for processes with $\alpha$-subexponential increments.

\begin{mytheo}\label{maj}
Let $(Z_x)_{x \in T}$ be a stochastic process indexed by a bounded set $T \subset \mathbb{R}^n$. 
Assume that the process has the uniform $\alpha$-subexponential increments, i.e., there exists $M \ge 0$ such that
\[
\|Z_x - Z_y\|_{\psi_\alpha} \le M \|x-y\|_2
\quad \text{for all } x,y \in T.
\]

Then we have
\begin{equation}\label{eq:thm2.7.14-mean-alpha}
     \textsf{E} \sup_{x,y \in T} |Z_x - Z_y|\leq   C(\alpha)M    \gamma_\alpha(T)
\end{equation}
Moreover, given $ u > 0$, we have

\begin{equation}\label{eq:thm2.7.14-tail-alpha}
    \textsf{P}\left( \sup_{s,t \in T} |Z_s - Z_t| \geq cM( u\cdot rad(T) + \gamma_\alpha(T)) \right) \leq  C\exp(-u^\alpha).
\end{equation}
\end{mytheo}

\begin{proof}
\emph{Step 1: Comparing $\delta_n$ with the Euclidean distance.}
By Proposition \ref{Lem_2.1} (f), for every random variable $\eta$ with finite
$\psi_\alpha$-norm and every $p\ge1$ we have
\[
  \big(\textsf{E}|\eta|^p\big)^{1/p}
  \le C\ p^{1/\alpha}\,\|\eta\|_{\psi_\alpha}.
\]
Applying this with
$\eta = Z_x - Z_y$ and $p=2^n$, we have
\[
  \|Z_x - Z_y\|_{2^n}
  = \big(\textsf{E}|Z_x - Z_y|^{2^n}\big)^{1/2^n}
  \le C\,2^{n/\alpha}\,\|Z_x - Z_y\|_{\psi_\alpha}
  \le C\,M\,2^{n/\alpha}\,\|x-y\|_2
\]
for all $x,y\in T$ and all $n\ge1$. 

Consequently, for every subset $A\subset T$,
\begin{equation}\label{eq:Delta-d2}
  \Delta_n(A)
  = \sup_{x,y\in A}\delta_n(x,y)
  \le CM\,2^{n/\alpha}\,rad(A),
  \qquad n\ge1.
\end{equation}

\medskip
\noindent
\emph{Step 2: Application of the generic chaining bound.}
Let $(\mathcal{A}_n)_{n\ge0}$ be an arbitrary admissible sequence of partitions of $T$
in the sense of Definition \ref{def2.3}. For each $t\in T$ we denote by $A_n(t)$
the unique element of $A_n$ containing $t$. Theorem \ref{thm:talagrand 2.7.14} applied to the centered
process $S_t := Z_t$ and the distances $\delta_n$ defined above gives the bound
\begin{equation}\label{eq:generic}
  \textsf{E}\sup_{x,y\in T}|Z_x - Z_y|
  \le C \sup_{t\in T}\sum_{n\ge0}\Delta_n\big(A_n(t)\big).
\end{equation}

Using \eqref{eq:Delta-d2}, we further obtain
\begin{equation}\label{eq:Delta<Gamma_T}
  \sup_{t\in T}\sum_{n\ge0}\Delta_n\big(A_n(t)\big)
  \le CM \sup_{t\in T}\sum_{n\ge0}2^{n/\alpha}\,
        rad\big(A_n(t)\big).
\end{equation}
Taking the infimum over all admissible sequences $(A_n)_{n\ge0}$ and recalling
the definition
\[
  \gamma_\alpha(T)
  := \inf\sup_{t\in T}\sum_{n\ge0}2^{n/\alpha}\,
     rad\big(A_n(t)\big),
\]
we conclude that
\[
  \textsf{E}\sup_{x,y\in T}|Z_x - Z_y|
  \le C(\alpha)\,M\,\gamma_\alpha(T).
\]   

Furthermore, observe that if $2^k \le u^2$, then
\[
\Delta_k(T) \le CM 2^{k/\alpha} \operatorname{rad}(T)
\le CM u^{2/\alpha} \operatorname{rad}(T).
\]
Combining this estimate with \eqref{eq:thm2.7.14-tail} and \eqref{eq:Delta<Gamma_T}, we obtain that for any $u>0$ and any admissible sequence $(\mathcal{A}_n)_{n\ge0}$ on $T$,
\[
\textsf{P}\Bigg(
\sup_{s,t \in T} |Z_s - Z_t|
\ge cM\Big(
u^{2/\alpha}\operatorname{rad}(T)
+ \sup_{t\in T}\sum_{n\ge0}2^{n/\alpha}\,
\operatorname{rad}\big(A_n(t)\big)
\Big)
\Bigg)
\le C \exp(-u^2).
\]
By a suitable reparameterization of $u$ and taking the infimum over all admissible sequences, we arrive at
\eqref{eq:thm2.7.14-tail-alpha}.

\end{proof}

\section{Proof of Theorem~\ref{thm:row}}
In this section, we give the proof of Theorem \ref{thm:row}.

We need the following results which are from Sambale \cite{sambale2023some}.
\begin{mytheo}\label{hason}{\cite[Theorem 1.1]{sambale2023some}}
For any \(\alpha \in (0,2]\), let \(X_1, \ldots, X_n\) be independent, centered random variables such that \(\|X_i\|_{\psi_\alpha} \leq K\) for any \(i\), and \(A = (a_{ij})\) be a symmetric matrix. Then, for any \(t \geq 0\),

\[
\textsf{P} \left( \left| X^TAX - \textsf{E}X^TAX \right| \geq t \right) \leq 2 \exp \left( -\frac{1}{C(\alpha)} \min \left( \frac{t^2}{K^4 \|A\|_{\text{HS}}^2}, \left( \frac{t}{K^2 \|A\|_{\text{op}}} \right)^{\frac{\alpha}{2}} \right) \right).
\]
\end{mytheo}

\begin{mytheo}\label{BB}\cite[Proposition 2.2]{sambale2023some}, 
 Let \( B \) be a fixed \( m \times n \) matrix and let \( X = (X_1, \ldots, X_n) \in \mathbb{R}^n \) be a centered random vector with independent  coordinates satisfying \( \textsf{E}X_i^2 = 1 \) and \( \|X_i\|_{\psi_{\alpha}} \leq K \) for some \(\alpha \in (0,2]\). 
Then for any \(t \geq 0\), we have

\begin{equation}
\textsf{P}(|\|BX\|_2 - \|B\|_{\text{HS}}| \geq tK^2\|B\|_{\text{op}}) \leq 2 \exp(-t^\alpha / C(\alpha)).
\end{equation}
and 
    $$\|\|BX\|_2 - \|B\|_{\text{HS}}\|_{\psi_\alpha} \leq C(\alpha) K^2\|B\|_{\text{op}}.$$

\end{mytheo}

\begin{mylem}\label{lem1}
 Let \( B \in \mathbb{R}^{l \times m} \) be a fixed matrix and  \( A \in \mathbb{R}^{m \times n} \) be a mean zero, isotropic and \(\alpha\)-subexponential matrix with \(\alpha\)-subexponential parameter \( K \). Then the random process

\[ Z_x := \|B A x\|_2 - \|B\|_{HS} \|x\|_2 \]
has \(\alpha\)-subexponential increments with

\[ \|Z_x - Z_y\|_{\psi_{\alpha}} \leq C(\alpha) K^{4/\alpha} \|B\|_{op} \|x - y\|_2, \ \forall x, y \in \mathbb{R}^n. \]
\end{mylem}

\begin{proof}
The statement is invariant under scaling for \(B\). So without loss of generality, we will assume \(B\) has operator norm \(\|B\|_{op} = 1\), $K\geq1$.

\textbf{Step 1:} Show \(\alpha\)-subexponential increments for \(x, y \in S^{n-1}\) on the unit sphere

Without loss of generality, assume \(x \neq y\) and define

\[ p := \textsf{P} \left( \frac{|Z_x - Z_y|}{\|x - y\|_2} \geq s \right) = \textsf{P} \left( \frac{\|B A x\|_2 - \|B A y\|_2}{\|x - y\|_2} \geq s \right). \]

We need to bound this tail probability by a $\alpha-subexponential tail$. Consider the following two cases:

\begin{itemize}
    \item \(s \geq 2\|B\|_{HS}\). Denote \(u := \frac{x-y}{\|x-y\|_2}\) and by triangle inequality we have
    \[ p \leq \textsf{P} \left( \frac{\|B A (x - y)\|_2}{\|x - y\|_2} \geq s \right) = \textsf{P} \left( \|B A u\|_2 \geq s \right) =: p_1. \]
    \item \(0 < s < 2\|B\|_{HS}\). 
    For $0<s\leq 1$, $p\leq 1\leq 2\exp(-cs^\alpha)$, where $c_1$ is a  small enough constant satisfies the inequality.
    
    For $1<s\leq 2\|B\|_{HS}$, we write \(p\) as
    \[ p = \textsf{P} \left( |Z| \geq s (\|B A x\|_2 + \|B A y\|_2) \right) \quad \text{where} \quad Z := \frac{\|B A x\|_2^2 - \|B A y\|_2^2}{\|x - y\|_2}. \]
\end{itemize}

Then we have

\[ p \leq \textsf{P} \left( |Z| \geq s \|B A x\|_2 \right) \leq \textsf{P} \left( \|B A x\|_2 \leq \frac{1}{2} \|B\|_{HS} \right) + \textsf{P} \left( |Z| > \frac{s}{2} \|B\|_{HS} \right) =: p_2 + p_3, \]

where \(p_2\) and \(p_3\) denote the first and second summand respectively.

\textbf{Bound for $p_1$.}

From \(s \geq 2\|B\|_{HS}\) we have

\[ p_1 = \textsf{P} \left( \|B A u\|_2 - \|B\|_{HS} \geq s - \|B\|_{HS} \right) \leq \textsf{P} \left( \|B A u\|_2 - \|B\|_{HS} \geq \frac{s}{2} \right). \]

Applying Theorem \ref{BB} to the random vector \(A u\) we get

\[ p_1 \leq 2 \exp \left( -{C(\alpha)} (\frac{s}{K^2 } )^\alpha\right). \]

\textbf{Bound for $p_2$.}

Applying Theorem \ref{BB} to the random vector \(A x\) and note that \(\|B\|_{HS} > \frac{1}{2}s\), we get

\[ p_2 \leq 2 \exp \left( -{C(\alpha)} \frac{\left( \|B\|_{HS} / 2 \right)^\alpha}{K^{2\alpha} } \right) \leq 2 \exp \left( -{C(\alpha)}( \frac{s}{4K^2 })^\alpha \right). \]

\textbf{Bound for $p_3$.}

Denote \(u := \frac{x-y}{\|x-y\|_2}\) and \(v := x + y\), then \(\langle u, v \rangle = 0\) since \(\|x\|_2 = \|y\|_2 = 1\). We can write \(Z\) as

\[ Z = \frac{\langle B A (x - y), B A (x + y) \rangle}{\|x - y\|_2} = \langle B A u, B A v \rangle. \]

Notice that

\[ 2 \langle B A u, B A v \rangle = \langle B A (u + v), B A (u + v) \rangle - \langle B A u, B A u \rangle - \langle B A v, B A v \rangle. \]

Let us also denote \(X_w := A w\) for \(w \in \mathbb{R}^n\), then from \(\textsf{E} X_w X_w^T = \|w\|_2^2 I_n\) we have

\[ \textsf{E} \|B X_w\|_2^2 = \textsf{E} \text{tr} (B^T B X_w X_w^T) = \text{tr} (B^T B \textsf{E} (X_w X_w^T)) = \|w\|_2^2 \|B\|_{HS}^2. \]

Thus we can further write \(Z\) as

\begin{align*}
Z &= \frac{1}{2} \|B X_{u+v}\|_2^2 - \frac{1}{2} \|B X_u\|_2^2 - \frac{1}{2} \|B X_v\|_2^2 \\
  &= \frac{1}{2} \left( \|B X_{u+v}\|_2^2 - \textsf{E} \|B X_{u+v}\|_2^2 \right) - \frac{1}{2} \left( \|B X_u\|_2^2 - \textsf{E} \|B X_u\|_2^2 \right) \\
  &\quad - \frac{1}{2} \left( \|B X_v\|_2^2 - \textsf{E} \|B X_v\|_2^2 \right) \\
  &= \frac{1}{2} Y_{u+v} - \frac{1}{2} Y_u - \frac{1}{2} Y_v,
\end{align*}

where the second equality uses the fact that \(Z\) is mean zero and in the last equality \(Y_w := \|B X_w\|_2^2 - \textsf{E} \|B X_w\|_2^2. \)  Therefore
\begin{align*}
p_3 &= \textsf{P}(|Y_{u+v} - Y_u - Y_v| > s \|B\|_{HS})\\
&\leq \textsf{P}(|Y_{u+v}| + |Y_u| + |Y_v| > s \|B\|_{HS})\\
&\leq \textsf{P} \left( |Y_{u+v}| \geq \frac{s}{2} \|B\|_{HS} \right) + \textsf{P} \left( |Y_u| + |Y_v| > \frac{s}{2} \|B\|_{HS} \right)\\
&\leq \textsf{P} \left( |Y_{u+v}| \geq \frac{s}{2} \|B\|_{HS} \right) + \textsf{P} \left( |Y_u| \geq \left( 1 - \frac{1}{8} \|v\|_2^2 \right) \frac{s}{2} \|B\|_{HS} \right)\\
&\quad+ \textsf{P} \left( |Y_v| > \frac{1}{8} \|v\|_2^2 \cdot \frac{s}{2} \|B\|_{HS} \right)\\
&=: p_4 + p_5 + p_6.
\end{align*}
We will bound \(p_4, p_5\) and \(p_6\) through the new Hanson-Wright inequality \eqref{hason}.

For any non-zero vector \(w\), define \(\bar{w} := \frac{w}{\|w\|_2}\). It is easy to see that \(X_w = \|w\|_2 X_{\bar{w}}\) and \(Y_w = \|w\|_2^2 Y_{\bar{w}}\). Also note that

\[ \|B^T B\|_{op} = \|B\|^2_{op} = 1, \quad \|B^T B\|_{HS} \leq \|B^T\| \|B\|_{HS} = \|B\|_{op} \|B\|_{HS} = \|B\|_{HS}, \]

so by Theorem 1.5 we have

\[ \textsf{P}(|Y_{\bar{w}}| \geq r) \leq 2 \exp \left( -{C(\alpha)} \min \left( \frac{t^2}{K^4 \|B\|_{HS}^2}, \left( \frac{t}{K^2 } \right)^{\frac{\alpha}{2}} \right) \right). 
\]

Hence for \(0 \leq t \leq \|w\|_2^2 \|B\|_{HS}^2\),

\[ \textsf{P}(|Y_w| \geq r) = \textsf{P} \left( |Y_{\bar{w}}| \geq \frac{r}{\|w\|_2^2} \right) \leq 2 \exp \left( -{C(\alpha)}  \left( \frac{t^2}{K^4 \|B\|_{HS}^2\|w\|_2^2} \right) \right)+2\exp\left( -{C(\alpha)}
\left( \frac{t}{K^2\|w\|_2^2 } \right)^{\frac{\alpha}{2}}\right).
\]

Now we apply Equation (16) to \(p_4, p_5\) and \(p_6\).

\begin{itemize}
    \item For \(p_4\). Since \(s < 2\|B\|_{HS}\) and \(\|u + v\|_2 = \sqrt{1 + \|v\|_2^2} \in [1, \sqrt{5})\), we can conclude that
    \[ \frac{s}{2} \|B\|_{HS} < \|B\|_{HS}^2 \leq \|u + v\|_2^2 \|B\|_{HS}^2 ,\]
    and therefore
    \[ p_4 \leq 2 \exp \left( -{C(\alpha)}\frac{s^2}{4 \|u + v\|_2^4 K^4} \right) + 2 \exp \left(-{C(\alpha)} \left(\frac{s\|B\|_{HS}}{2 K^2\|u + v\|_2^2}\right)^{\frac{\alpha}{2}} \right). \]

    Since $\|B\|_{HS}\geq\|B^TB\|_{HS}\geq\|B^TB\|_{op}=\|B\|^2_{op}=1$ , and \(\|u + v\|_2  \in [1, \sqrt{5})\)

    \[ p_4 \leq 2 \exp \left( -{C(\alpha)}\frac{s^2}{ K^4} \right) + 2 \exp \left(-{C(\alpha)} \frac{s^\alpha}{ K^\alpha} \right). \]
    
    \item For \(p_5\). Notice that \(\|u\|_2 = 1\) and \(1 - \frac{1}{8} \|v\|_2^2 \in (\frac{1}{2}, 1]\), so
    \[ p_5 \leq 2 \exp \left( -{C(\alpha)}\frac{s^2}{ K^4} \right) + 2 \exp \left(-{C(\alpha)} \frac{s^\alpha}{ K^\alpha} \right). \]
    \item For \(p_6\). If \(v = 0\) (i.e. \(x = -y\)), then \(p_6 = \textsf{P}(0 > 0) = 0\). Now assume \(v \neq 0\), since $\|v\|_2\leq 2,\|B\|_{HS}\geq1$
    
    then like $p_4$ we have
     \[ p_6 \leq 2 \exp \left( -{C(\alpha)}\frac{s^2}{ K^4} \right) + 2 \exp \left(-{C(\alpha)} \frac{s^\alpha}{ K^\alpha} \right). \]
\end{itemize}
\textbf{Combining the previous inequalities, we get the following bounds for  } $p$.

So far, we have shown that

\[ p \leq \max\{p_1, p_2 + p_3\} \quad \text{and} \quad p_3 \leq p_4 + p_5 + p_6, \]

where \(p_i \leq 4 \exp \left( \frac{-cs^\alpha}{K^4} \right)\) for some absolute constant \(c\) and \(1 \leq i \leq 6\). Note that \(p \leq 1\) and the inequality \(\min\{1, 16e^{-x}\} \leq 2e^{-x/4}\) . So we get

\[ p \leq \min \left\{ 1, 16 \exp \left( \frac{-cs^\alpha}{K^4 } \right) \right\} \leq 2 \exp \left( \frac{-cs^\alpha}{4K^4 } \right). \]

\textbf{Step 2:} Show $\alpha-subexponential$ increments for all \(x\) and \(y\)

Without loss of generality, we can assume \(\|x\|_2 = 1\) and \(\|y\|_2 \geq 1\). Let \(\bar{y} := \frac{y}{\|y\|_2}\) be the projection of \(y\) onto unit ball, then by triangle inequality,
\[ \|Z_x - Z_y\|_{\psi_\alpha} \leq C(\alpha)\left(\|Z_x - Z_{\bar{y}}\|_{\psi_\alpha} + \|Z_y - Z_{\bar{y}}\|_{\psi_\alpha}\right) =: R_1 + R_2. \]
Here \(R_1\) it is bounded by \(C(\alpha)K^{\frac{\alpha}{4}} \|x - \bar{y}\|_2\) since \(x, \bar{y} \in S^{n-1}\), and
\[ R_2 = \|(\|y\|_2 - 1) Z_{\bar{y}}\|_{\psi_\alpha} = \|y - \bar{y}\|_2 \|Z_{\bar{y}}\|_{\psi_\alpha} \leq C(\alpha)K^{\frac{\alpha}{4}} \|x - \bar{y}\|_2, \]
where the first equality uses \(Z_y = \|y\|_2 Z_{\bar{y}}\). The second equality is true since \(\|y\|_2 - 1 = \|y - \bar{y}\|_2\) and the last inequality follows from Theorem \ref{BB}. Combining these bounds, we get
\[ \|Z_x - Z_y\|_{\psi_\alpha} \leq C(\alpha)K^{\frac{\alpha}{4}}    (\|x - \bar{y}\|_2 + \|y - \bar{y}\|_2). \]

Finally, note that \(\|x\|_2 = 1\), so by non-expansiveness of projection, \(\|x - \bar{y}\|_2 \leq \|x - y\|_2\), and by definition of projection, \(\|y - \bar{y}\|_2 \leq \|y - x\|_2\). This completes the proof.
\end{proof}

\begin{proof}[Proof of Theorem~\ref{thm:row}]
Let  $Z_x := \|BAx\|_2 -\| B\|_{HS}\|x\|_2.$

For the expectation bound, take an arbitrary $ y \in T$,  then from the triangle inequality, we have 
$$\textsf{E} \sup_{t \in T} |Z_t| \leq \textsf{E} \sup_{t \in T} |Z_t - Z_y| + \textsf{E} |Z_y|. $$
Using Lemma \ref{lem1} and Theorem \ref{maj}, we get
$$\textsf{E} \sup_{t \in T} |Z_t - Z_y| \leq \textsf{E} \sup_{t, y \in T} |Z_t - Z_y| \lesssim_\alpha M \gamma_\alpha(T,d_2). $$
Using Lemma \ref{lem1} and property (f) in Proposition~\ref{Lem_2.1}, we get
$$\textsf{E} |Z_y| \lesssim \|Z_y\|_{\psi_\alpha} = \|Z_y - Z_0\|_{\psi_\alpha} \lesssim K^{4/\alpha}  \, \|B\|_{op}\,\|y\|_2. $$
Therefore  $\textsf{E} \sup_{t \in T} |Z_t| \leq C(\alpha) K^{4/\alpha}\|B\|_{op}(\gamma_\alpha(T,d_2)+\|y\|_2).$

For the high probability bound, notice that the result is trivial when $ u < 1.$ 
When $ u \geq 1,$  fix an arbitrary  $y \in T$ and use triangle inequality again to get 
$$\sup_{t \in T} |Z_t| \leq \sup_{t \in T} |Z_t - Z_y| + |Z_y| \leq \sup_{t, t' \in T} |Z_t - Z_{t'}| + |Z_y|.$$ 
 
 Let\[p_1=\textsf{P}\left( \sup_{t, t' \in T} |Z_t - Z_{t'}| \geq c(\alpha)K^{4/\alpha}( u\cdot rad(T) + \gamma_\alpha(T)) \right)  ,\]
 \[p_2=\textsf{P}\left( |Z_y|\ge c(\alpha)K^{4/\alpha}u\cdot rad(T)\right).\]
 
By Theorem~\ref{maj}, we obtain the estimate
\[
p_1 \le C \exp(-u^\alpha).
\]
On the other hand, property (b) in Proposition~\ref{Lem_2.1} yields
\[
p_2 \le 2 \exp(-u^\alpha).
\]
Therefore, we have
\[
\textsf{P}\Bigg(
\sup_{x \in T} \bigl| \|BAx\|_2 - \|B\|_{HS} \|x\|_2 \bigr|
\ge c(\alpha)K^{4/\alpha}\, \|B\|_{op}
\bigl( \gamma_\alpha(T) + u\, \operatorname{rad}(T) \bigr)
\Bigg)
\le p_1+p_2\le C \exp(-u^\alpha).
\]

\end{proof}

\section{Proof of Theorem~\ref{thm:column}}
Similar to the proof of Theorem~\ref{thm:row}, our goal is to show that the process
\[
\bigl\{\|Ax\|_2 - \|x\|_2\bigr\}_{x \in T}
\]
has uniformly $\alpha$-subexponential increments. 
To this end, we first consider a  more simple special case.

\begin{mylem}\label{lem:Ax-1}

 Let $A$ be an $m \times n$ matrix whose columns $A_i$ are independent,  $\alpha$-subexponential random vectors in $\mathbb{R}^m$ satisfying $\|A_i\|_2 = 1$ almost surely and $\|\langle A_i,x\rangle\|_{\psi_\alpha} \leq K$ for all $x\in S^{m-1}$. Then we have
 \[\|\|Ax\|_2-1\|_{\psi_\alpha}\leq C(\alpha) K.\]

\end{mylem}

\begin{proof}
With loss of generality, we can assume $K$ is large enough. Since
$$\|Ax\|_2^2 = \sum_{i=1}^{n} x_i^2 \|A_i\|_2^2 + \sum_{i \neq j} x_i x_j \langle A_i, A_j \rangle,$$
we have
$$\|Ax\|_2^2 - 1 = \sum_{i \neq j} x_i x_j \langle A_i, A_j \rangle.$$
For any $p\geq1$,
\begin{align}
    \left( \textsf{E} \left| \|Ax\|_2^2 - 1 \right|^p \right)^{\frac{1}{p}}& = \left( \textsf{E} \left| \sum_{i \neq j} x_i x_j \langle A_i, A_j \rangle \right|^p \right)^{\frac{1}{p}}\nonumber\\
    &\leq \left( \textsf{E} \left| 4 \sum_{i \neq j} x_i x_j \langle A_i, A'_j \rangle \right|^p \right)^{\frac{1}{p}} \nonumber\\
    &= \left( \textsf{E} \left| 4 \langle Ax, A'x \rangle \right|^p \right)^{\frac{1}{p}}.
\end{align}
Here $A'$ is an independent copy of $A$.
The second inequality here can refer to \cite{vershynin2018high} theorem 6.1.1 :the decoupling technology.

Note that $\langle A_i, A'_j \rangle$ is conditional $\alpha$-subexponential, and the norm of conditional $\alpha$-subexponential is bounded by $CK\|AX\|_2$.

Thus we have 
\begin{align*}
    \left( \textsf{E} \left| 4 \langle Ax, A'x \rangle \right|^p \right)^{\frac{1}{p}}&\leq\left( \textsf{E} \left(\|Ax\|_2 CK p^{\frac{1}{\alpha}} \right)^p \right)^{\frac{1}{p}}\\
    &= CK p^{\frac{1}{\alpha}}\| \|Ax\|_2 \|_p\\
    &= CK p^{\frac{1}{\alpha}}\sqrt{ \|\|Ax\|_2^2\|_{\frac{2}{p}}}.
\end{align*}

We can easily get
\begin{align*}
\textsf{E}\|Ax\|_2^2 &= \textsf{E}\sum_{i=1}^{n} x_i^2 \|A_i\|_2^2 + \textsf{E}\sum_{i \neq j} x_i x_j \langle A_i, A_j \rangle\\
&=\textsf{E}\sum_{i=1}^{n} x_i^2 \|A_i\|_2^2\\
&=\sum_{i=1}^{n} x_i^2 \\
&=1.
\end{align*}

For $p>0$, we define $f(p):=\|\|Ax\|^2 _2-1\|_p$, then we have

	(i)$f(p)\leq CKp^{\frac{1}{\alpha}}\sqrt{f(\frac{p}{2})+1},$
    
    (ii)$\leq f(2)\leq CK2^{\frac{1}{\alpha}}\sqrt{\textsf{E}\|Ax\|_2^2}=CK2^{\frac{1}{\alpha}}.$

Let $CK=c_0$. As K is large enough, without loss of generality, we may assume that $c_0\geq 1.$

Consider the following sequence:
$$\begin{cases} 
a_{n+1} = c_0 \cdot 2^{\frac{n+1}{\alpha}} \sqrt{a_n + 1}, & n \geq 2, \quad (1) \\
a_2 = c_0 \cdot 2^{\frac{2}{\alpha}} .\quad (2)
\end{cases}$$

Obviously,  $a_n \geq 1 \text{ for all } n \geq 2$.
Let \( b_n = \log_2 a_n \).

From (1), we get:
\begin{align*}
    b_{n+1}& = \log_2 \left( c_0 2^{\frac{n+1}{\alpha}} \sqrt{2^{b_n} + 1} \right)\\
    &= \log_2 \left( \frac{b_0}{2^2} \cdot 2^{\frac{n+1}{\alpha}} \cdot c_0 \sqrt{\frac{2^{b_{n+1}}}{2^{b_n}}} \right)\\
    &\leq \frac{1}{2} b_n + \frac{n+1}{\alpha} + \log_2 \left( 2c_0\right).
\end{align*}

From (2), we get:
\[ b_2 \leq \frac{2}{\alpha} + \log_2 (2c_0). \]

Thus, we have:
\begin{align*}
     b_n &\leq \frac{1}{2} b_{n-1} + \frac{n}{\alpha} + \log_2 (2c_0)\\
     &\leq \frac{1}{4} b_{n-2} + \frac{n}{\alpha} + \log_2 (2c_0) + \frac{1}{2} \left( \frac{n-1}{\alpha} + \log_2 (2c_0) \right)\\
     &\leq \ldots\\
     &\leq \left( \frac{1}{2} \right)^{n-2} b_2 + \sum_{j=0}^{n-3} \left( \frac{1}{2} \right)^j \left( \frac{n-j}{\alpha} + \log_2 (2c_0) \right) \\
     &\leq \sum_{j=0}^{n-2} \left( \frac{1}{2} \right)^j \left( \frac{n-j}{\alpha} + \log_2 (2c_0) \right)\\
     &\leq 2 \frac{n}{\alpha} + 2 \log_2 (2c_0) .
\end{align*}




Thus, correspondingly,
\[ a_n = 2^{b_n} \leq 4c_0^2 \cdot 2^{\frac{2n}{\alpha}}. \]

Combining (i), (ii), (1)and (2), it is easy to see:
\[ f(2^n) \leq 4c_0^2 \cdot 2^{\frac{2n}{\alpha}} .\]

That is,
\[ \left\| \|Ax\|_2^2 - 1 \right\|_{2^n} \leq 4c_0^2 \cdot 2^{\frac{2n}{\alpha}} \]
for all \( n \geq 1 \).

Since
\[ \|Ax\|_2^2 - 1 = (\|Ax\|_2 - 1)(\|Ax\|_2 + 1), \]
and
\[ |\|Ax\|_2 - 1| \leq \|Ax\|_2 + 1 ,\]
it follows that
\[ |\|Ax\|_2 - 1|^2 \leq |\|Ax\|_2^2 - 1 |.\]

Thus
\[ \left( \textsf{E} \left| \|Ax\|_2 - 1 \right|^{2^{n+1}} \right)^{\frac{1}{2^n}} \leq 4c_0^2 \cdot 2^{\frac{2n}{\alpha}} \]
\[ \Rightarrow \left\| \|Ax\|_2 - 1 \right\|_{2^{n+1}} \leq 2c_0 \cdot 2^{\frac{n}{\alpha}} \leq 2c_0 \cdot 2^{\frac{n+1}{\alpha}} \]
for all \( n \geq 1 \).
Therefore, for all \( p \geq 1 \), we define \( n_p := \min \{ n : 2^n \geq 4p \} \).

Then we have 
\[ \left\| \|Ax\|_2 - 1 \right\|_p \leq \left\| \|Ax\|_2 - 1 \right\|_{2^{n_p}} \leq 2c_0 \cdot 2^{\frac{n_p}{\alpha}} \]
\[ \leq 2c_0 \cdot (8p)^{\frac{1}{\alpha}} . \]
That is,
\[ \left\| \|Ax\|_2 - 1 \right\|_{p} \leq C(\alpha) K p^{\frac{1}{\alpha}}. \]
(Recall \( c_0 = CK \))

Using property (f) in Proposition~\ref{Lem_2.1}, we have
\[
\|\|Ax\|_2-1\|_{\psi_\alpha}\leq C(\alpha) K.
\]
\end{proof}

Somewhat surprisingly, the analysis of this special case already suffices to establish the general result:
\begin{mytheo}\label{thm:alpha_increase}
Assume \( Z_x = \|Ax\|_2 - \|x\|_2 \), then \( \|Z_x - Z_y\|_{\psi_\alpha} \leq C(\alpha) K \|x - y\|_2 \) for all \( x, y \in \mathbb{R}^n \).
\end{mytheo}

\begin{proof}
Due to the homogeneity, we can assume \( K \geq 1 .\) We have
\begin{align*}
    \left\| \frac{Z_x - Z_y}{\|x - y\|_2} \right\|_{\psi_\alpha}
     &= \left\| \frac{(\|Ax\|_2 - \|Ay\|_2) - (\|x\|_2 - \|y\|_2)}{\|x - y\|_2} \right\|_{\psi_\alpha}
    \\&\leq C(\alpha)\left(\left\| \frac{\|Ax\|_2 - \|Ay\|_2}{\|x - y\|_2} \right\|_{\psi_\alpha} + \left\| \frac{\|x\|_2 - \|y\|_2}{\|x - y\|_2} \right\|_{\psi_\alpha}\right) .
\end{align*}

By the triangle inequality, we have:
\[ \left| \|Ax\|_2 - \|Ay\|_2 \right| \leq \|A(x-y)\|_2 ,\]
\[ \left| \|x\|_2 - \|y\|_2 \right| \leq \|x-y\|_2 .\]

Therefore:
\begin{align*}
\left\| \frac{Z_x - Z_y}{\|x -y\|_2}\right\|_{\psi_\alpha}&\leq C(\alpha)\left(\left\| \frac{\|A(x-y)\|_2}{\|x-y\|_2} \right\|_{\psi_\alpha} + 1 \right)\quad \\
 &\leq C(\alpha)\left(\left\| \frac{\|A(x-y)\|_2 - \|x-y\|_2}{\|x-y\|_2} \right\|_{\psi_\alpha} + 2 \right)\leq C(\alpha)K. \text{ (for } K \geq 1)
\end{align*}

\end{proof}

The remainder of the proof follows the same lines as that of Theorem~\ref{thm:row} and is therefore omitted.

\section{Applications}

\subsection{Johnson-Lindenstrauss Lemma}

A direct consequence of our results is that all matrices satisfying the assumptions of
Theorem~\ref{thm:row} or Theorem~\ref{thm:column} serve as Johnson--Lindenstrauss embeddings for dimension reduction.
We record the corresponding Johnson--Lindenstrauss lemma below.

\begin{mylem}
    Let $A\in \mathbb{R}^{m\times n}$ satisfy the conditions in Theorem \ref{thm:row}. For any $\delta>0$ and $0\le \varepsilon\le 1$, if 
    \[
    m\ge C(\alpha)K^{\frac{8}{\alpha}}\varepsilon^{-2}(\log \frac{1}{\delta})^{\frac{2}{\alpha}},
    \]
    then for any $x,y\in \mathbb{R}^n$, with probability at least $1-\delta$, we have
    \[
    (1-\varepsilon)\|x-y\|_2\le \frac{1}{\sqrt{m}}\|A(x-y)\|_2\le (1+\varepsilon)\|x-y\|_2.
    \]
\end{mylem}

\begin{proof}
    By scaling, we can assume $\|x-y\|_2=1$. By Theorem ~\ref{maj} with $B=I_m$, we have
    \[
    \|\frac{1}{\sqrt{m}}\|A(x-y)\|_2-\|x-y\|_2\|_{\psi_\alpha}\le C(\alpha)K^{4/\alpha}m^{-1/2}.
    \]
    Combining with the property (b) in Proposition \ref{Lem_2.1}, we can get the result immediately.
\end{proof}

Similarly, we have the following lemma.

\begin{mylem}
    Let $A\in \mathbb{R}^{m\times n}$ satisfy the conditions in Corollary \ref{coro:column} with $\lambda=\sqrt{m}$. For any $\delta>0$ and $0\le \varepsilon\le 1$, if 
    \[
    m\ge C(\alpha)K^{2}\varepsilon^{-2}(\log \frac{1}{\delta})^{\frac{2}{\alpha}},
    \]
    then for any $x,y\in \mathbb{R}^n$, with probability at least $1-\delta$, we have
    \[
    (1-\varepsilon)\|x-y\|_2\le \frac{1}{\sqrt{m}}\|A(x-y)\|_2\le (1+\varepsilon)\|x-y\|_2.
    \]
\end{mylem}

\subsection{R.I.P. of $\alpha$-subexponential random matrices}

A vector $x \in \mathbb{R}^n$ is said to be $s$-sparse if it has at most $s$ nonzero entries, that is,
\begin{equation}
\|x\|_0 := \bigl|\{\, l : x_l \neq 0 \,\}\bigr| \le s .
\end{equation}
In compressed sensing, one seeks to reconstruct such a vector from linear observations of the form
\begin{equation}
y = \Phi x ,
\end{equation}
where $\Phi \in \mathbb{R}^{m \times n}$ is a given sensing matrix. 
Since the system is fully determined only when $m=n$, the problem is of interest primarily in the undersampled regime $m \ll n$.

A widely studied reconstruction method is based on $\ell_1$-minimization, namely,
\begin{equation}
\min_{z \in \mathbb{R}^n} \|z\|_1
\quad \text{subject to} \quad
\Phi z = y .
\end{equation}
This naturally raises the question of when the solution to this convex program coincides with the original sparse signal.

A fundamental sufficient condition is the restricted isometry property (R.I.P.). 
An $m \times n$ matrix $A$ is said to satisfy the R.I.P. of order $s$ with constant $\delta \in (0,1)$ if
\begin{equation}\label{eq:R.I.P.}
(1-\delta)\|x\|_2^2 \le \|Ax\|_2^2 \le (1+\delta)\|x\|_2^2
\end{equation}
holds for all $s$-sparse vectors $x \in \mathbb{R}^n$.
The restricted isometry constant $\delta_s$ is defined as the smallest $\delta$ for which \eqref{eq:R.I.P.} holds.

Dai et al.~\cite{dai2025log} studied the restricted isometry property for random matrices with $\alpha$-subexponential tails.
By contrast, the conclusions obtained here concentrate on particular instances arising from our two main theorems.

\begin{myprop}
    Let $A\in \mathbb{R}^{m\times n}$ satisfy the conditions in Theorem \ref{thm:row}. For any $\delta\in (0,1)$, if 
    \[
    m\ge C(\alpha)\delta^{-2}K^{\frac{8}{\alpha}}\Bigl((s\log\frac{en}{s})^\frac{1}{\alpha}+u\Bigr)^{2},
    \]
    then with probability at least $1-C\exp(-u^\alpha)$, we have
    \[
    \delta_s\le \delta.
    \]
\end{myprop}

\begin{proof}
    By homogeneity, it suffices to verify \eqref{eq:R.I.P.} for all vectors
\[
x \in T := \Sigma_{s}^{2n} \cap \mathbb{S}^{n-1},
\]
where $\Sigma_{s}^{2n}$ denotes the set of $s$-sparse vectors in $\mathbb{R}^{2n}$.
Combining this observation with Corollary~\ref{coro:row}, we obtain that, with probability at least
$1 - C e^{-u^\alpha}$,
\begin{equation}\label{eq:RIP_row}
\sup_{x \in T}
\left| \frac{1}{\sqrt{m}}\|Ax\|_2 - 1 \right|
\le
C(\alpha)\, m^{-1/2} K^{4/\alpha}
\bigl( \gamma_\alpha(T) + u\, \operatorname{rad}(T) \bigr).
\end{equation}
Consequently, it remains to estimate the $\gamma_\alpha$-functional of the set $T$.

A general upper bound for $\gamma_\alpha(T)$ is given by
\[
\gamma_\alpha(T)
\lesssim
\int_0^\infty \bigl( \log N(T,d,u) \bigr)^{1/\alpha} \, du,
\]
see, for instance, Chapter~3 of Talagrand~\cite{Talagrand2014} (where the case $\alpha=2$ is treated).
On the other hand, a standard volumetric argument yields
\[
N\bigl(T, \|\cdot\|_2, u\bigr)
\le
\Bigl(\frac{e\, 2n}{s}\Bigr)^s (1+2/u)^s .
\]
Combining these estimates, we obtain
\[
\gamma_\alpha(T)
\le
C(\alpha) s^{1/\alpha}
\left(
\log^{1/\alpha}\!\Bigl(\frac{e n}{s}\Bigr)
+ \int_0^1 \log^{1/\alpha}\!\Bigl(1+\frac{2}{u}\Bigr)\, du
\right)
\le
C(\alpha) \bigl( s \log (e n / s) \bigr)^{1/\alpha}.
\]
Substituting this bound into \eqref{eq:RIP_row} yields the desired conclusion.

\end{proof}

Similarly, we have the following proposition.

\begin{myprop}
    Let $A\in \mathbb{R}^{m\times n}$ satisfy the conditions in Theorem \ref{thm:column}. Then for any $\delta\in (0,1)$, if 
    \[
    m\ge C(\alpha)\delta^{-2}K^{2}\Bigl((s\log\frac{en}{s})^\frac{1}{\alpha}+u\Bigr)^{2},
    \]
    then with probability at least $1-C\exp(-u^\alpha)$, we have
    \[
    \delta_s\le \delta.
    \]
\end{myprop}

\subsection{Normalizing the columns of an $\alpha$-subexponential matrix}
We turn to the study of random matrices whose columns are independent, isotropic and $\alpha$-subexponential. 

These assumption ensures the Euclidean norms of the columns concentrate around $\sqrt{m}$. 
Motivated by this fact, we rescale each column to lie on the sphere $\sqrt{m}\mathbb{S}^{n-1}$. 
However, in order to prevent excessive amplification caused by columns of unusually small norm, it is necessary to restrict attention to the event on which all column norms are uniformly bounded from below. 
This requirement is formalized by the event $\mathcal{F}$ defined below.

\begin{mycoro}
\label{cor:column-normalization}
Let $A \in \mathbb{R}^{m \times n}$. Assume $\alpha\in(0,2]$. Let the columns of $A$ be independent, symmetric, isotropic and have \(\psi_\alpha\)-norm  (or quasi-norm when $\alpha<1$) bounded by $K$.
Assume that
\[
m\ge C(\alpha) K^4(\log n)^{2/\alpha} .
\]
Denote by $A_i$ the $i$th column of $A$. Define the event
\[
\mathcal{F}
:=
\Bigl\{
\min_{1 \le i \le n} \|A_i\|_2 \ge \tfrac{\sqrt{m}}{2}
\Bigr\}.
\]
Then
\[
\textsf{P}(\mathcal{F})
\ge
1-2\exp\!\Bigl(-c(\alpha)\,\frac{m^{\alpha/2}}{K^{2\alpha}}\Bigr).
\]

Conditioned on the event $\mathcal{F}$, define the column-normalized matrix $\widetilde{A}$ by
\[
\widetilde{A}_i
:=
\frac{\sqrt{m}}{\|A_i\|_2}\, A_i,
\qquad i = 1,2,\dots,n.
\]
Then
\[
\textsf{E}\Bigl[
\sup_{x \in T}
\bigl|
\|\widetilde{A}x\|_2 - \sqrt{m}\|x\|_2
\bigr|
\;\Big|\;
\mathcal{F}
\Bigr]
\le
C(\alpha) K (\gamma_\alpha(T)+rad(T)).
\]
Moreover, still conditional on $\mathcal{F}$, for every $u>0$, with probability at least
$1 - 3 \exp(-u^2)$,
\[
\sup_{x \in T}
\bigl|
\|\widetilde{A}x\|_2 - \sqrt{m}\|x\|_2
\bigr|
\le
C(\alpha) K \bigl( \gamma_\alpha(T) + u \cdot \operatorname{rad}(T) \bigr).
\]
\end{mycoro}

\begin{proof}

We first introduce a high-probability event on which all columns of $A$ have Euclidean norms bounded away from zero. Conditioning on this event, we verify that the normalized matrix $\widetilde A$ has independent, mean-zero, subgaussian columns with controlled $\psi_2$-norm, and then invoke Corollary~\ref{coro:column}.

\medskip
\noindent
\textit{Step 1: a uniform lower bound on column norms.}
Let $A_1,\dots,A_n$ denote the columns of $A$, and define
\[
F \;:=\; \Bigl\{ \min_{1\le i\le n}\|A_i\|_2 \ge \tfrac{\sqrt m}{2} \Bigr\}.
\]
By Theorem \ref{BB}, for each $i\in\{1,\dots,n\}$ one has
\[
\bigl\|\|A_i\|_2-\sqrt m \bigr\|_{\psi_\alpha}\;\le\; C(\alpha) K^2,
\]
and therefore
\[
\textsf P\Bigl(\|A_i\|_2\le\tfrac{\sqrt m}{2}\Bigr)
=\textsf P\Bigl(\|A_i\|_2-\sqrt m\le -\tfrac{\sqrt m}{2}\Bigr)
\le 2\exp\!\Bigl(-c(\alpha)\,\frac{m^{\alpha/2}}{K^{2\alpha}}\Bigr).
\]
Applying the union bound yields
\[
\textsf P(F)
\;\ge\; 1-2n\exp\!\Bigl(-c(\alpha)\,\frac{m^{\alpha/2}}{K^{2\alpha}}\Bigr)
\;=\; 1-2\exp\!\Bigl(\log n-c(\alpha)\,\frac{m^{\alpha/2}}{K^{2\alpha}}\Bigr).
\]
Under the standing assumption $m\ge C(\alpha) K^4(\log n)^{2/\alpha}$, the $\log n$ term can be absorbed into the constant, and hence
\[
\textsf P(F)\;\ge\; 1-2\exp\!\Bigl(-c(\alpha)\,\frac{m^{\alpha/2}}{K^{2\alpha}}\Bigr)\;\ge\;\tfrac12.
\]

\medskip
\noindent
\textit{Step 2: conditional properties of the normalized columns.}
On the event $F$ the columns are nonzero, so we may define the column-normalized matrix $\widetilde A$ by
\[
\widetilde A_i \;:=\; \frac{\sqrt m}{\|A_i\|_2}\,A_i,\qquad i=1,\dots,n.
\]
Clearly, $\|\widetilde A_i\|_2=\sqrt m$ for all $i$. Moreover, conditioning on $F$ gives
\begin{equation}\label{eq:psi2-bound}
\|\widetilde A_i\mid F\|_{\psi_\alpha}
=\Bigl\|\frac{\sqrt m}{\|A_i\|_\alpha}A_i \,\Bigm|\, F\Bigr\|_{\psi_\alpha}
\le 2\,\|A_i\mid F\|_{\psi_\alpha}
\le C\,\|A_i\|_{\psi_\alpha}.
\end{equation}
The first inequality uses $\sqrt m/\|A_i\|_2\le 2$ on $F$. The second inequality follows from the fact that conditioning on an event of probability at least $1/2$ inflates the $\alpha$-subexponential norm by at most an absolute constant.
Consequently, by the assumption $\|A_i\|_{\psi_\alpha}\le K$, we have 
\[
\|\widetilde A_i\mid F\|_{\psi_\alpha}\;\le\; C K.
\]

\medskip
\noindent
\textit{Step 3: conditional centering and independence.}
Since the entries of $A_i$ are symmetric, their signs are independent of their magnitudes; hence, after conditioning on $F$ (which only depends on $\|A_i\|_2$), the conditional distribution of $\widetilde A_i$ remains symmetric. In particular,
\[
\textsf E[\widetilde A_i\mid F]=0.
\]
Finally, $F$ factorizes across columns (it is the intersection of events depending on individual columns), and the columns of $A$ are independent. Therefore the columns of $\widetilde A$ remain independent conditional on $F$.

\medskip
\noindent
\textit{Conclusion.}
Conditional on $F$, the columns $\widetilde A_1,\dots,\widetilde A_n$ are independent, mean-zero, subgaussian vectors satisfying
\[
\|\widetilde A_i\|_2=\sqrt m,\qquad \|\widetilde A_i\mid F\|_{\psi_\alpha}\le CK,\qquad \textsf E[\widetilde A_i\mid F]=0.
\]
Thus we are precisely in the setting of Corollary~\ref{coro:column} with $\lambda=\sqrt m$, and the desired claim follows.
\end{proof}

\textbf{Acknowledgment}
The work of Xuanang Hu (corresponding author) and Hanchao Wang were supported by the National Key R\&D Program of China (No.2024YFA1013501), the National Natural Science Foundation of China (No. 12571162), and Shandong Provincial Natural Science Foundation (No. ZR2024MA082). The work of Vladimir V. Ulyanov was conducted within the framework of the HSE University Basic Research Programs and within the program of the Moscow Center for Fundamental and Applied Mathematics, Lomonosov Moscow State University.

\textbf{Author contributions}

These authors contributed equally to this work.

\textbf{Data availability}

No data was used for the research described in the article.

\printbibliography

\end{document}